\documentclass[a4paper,11pt]{article}
\usepackage{amsthm,amsmath}
\usepackage{amssymb}
\RequirePackage{natbib}
\RequirePackage[hidelinks]{hyperref}
\usepackage{mathtools}
\usepackage{graphicx}
\usepackage{color}
\usepackage{tabularx}
\usepackage{float}
\usepackage{array}
\usepackage{multirow}

\numberwithin{equation}{section}
\theoremstyle{plain}
\newtheorem{definition}{Definition}[section]
\newtheorem{proposition}[definition]{Proposition}
\newtheorem{lemma}[definition]{Lemma}

\newtheorem{remark}[definition]{Remark}

\usepackage{graphicx}

\begin{document}
\title{Robust and non asymptotic estimation of probability weighted moments with application to extreme value analysis}
\author{Anna B\textsc{en}-H\textsc{amou}$^1$, Philippe N\textsc{aveau}$^{2}$ and  Maud T\textsc{homas}$^1$
 }

\date{}

\maketitle

\begin{center} 
{\small$^1$ Sorbonne Universit\'e, CNRS, Laboratoire de Probabilit\'es, Statistique et Mod\'elisation, \\LPSM, 4 place Jussieu, F-75005 Paris, France,\\
$^2$ Laboratoire des Sciences du Climat et de l'Environnement, UMR8212 CEA-CNRS-UVSQ, IPSL \& Universit\'e Paris-Saclay, 91191 Gif-sur-Yvettes, France\\
E-mails: anna.ben\_hamou@sorbonne-universite.fr, \\
philippe.naveau@lsce.ipsl.fr, \\maud.thomas@sorbonne-universite.fr}
\end{center}

\begin{abstract}
In extreme value theory and other related risk analysis fields, probability weighted moments (PWM) have been frequently used to estimate the parameters of classical  extreme value distributions.
This method-of-moment technique can be applied  when second moments are finite, a reasonable assumption in many environmental domains like climatological and hydrological studies. Three advantages of PWM estimators can be put forward:  their simple interpretations, their rapid numerical implementation and their close connection to the well-studied class of  $U$-statistics. Concerning the later, this connection leads to precise asymptotic properties, but non asymptotic bounds have been lacking when off-the-shelf techniques (Chernoff method) cannot be applied, as  exponential moment assumptions  become unrealistic in many extreme value settings. In addition, large values analysis    is not immune to the undesirable  effect of outliers, for example,  defective readings in satellite measurements or possible anomalies in climate model runs.
Recently, the treatment of outliers  has sparked some interest in extreme value theory, but results about finite sample bounds in a robust extreme value theory  context are yet to be found, in particular for  PWMs or  tail index estimators. 
In this work, we propose a new class of  robust PWM estimators, inspired by the median-of-means framework of~\citet{devroye2016sub}. This class of robust estimators is  shown to satisfy a sub-Gaussian inequality when the assumption  of  finite second moments holds. 
Such non asymptotic bounds are also derived under the general contamination model. 
Our  main proposition confirms theoretically the  trade-off between efficiency and robustness pointed out by \citet{Brazauskas00}.
Our simulation study indicates that, while classical estimators of PWMs can be highly sensitive to outliers,  our new approach remains weakly affected by the degree contamination.
\end{abstract}


\section{Introduction}

Let $X$ be an integrable real-valued random variable with cumulative distribution function  $F$. The probability weighted moments (PWMs) of $X$ are defined as
\[
E\left(XF(X)^{r}\overline{F}(X)^{s}\right) 
\]
where $r$ and $s$ are non-negative integers, and $\overline F = 1-F$ denotes the survival function associated with $F$. 
The use of these moments have been motivated by hydrologists and applied statisticians \citep[see, e.g.][]{Hosk:Wall:87,Land:Mata:Wall:79, greenwood1979probability}.They also appear naturally in the expression of the parameters of several distributions used in extreme value theory~\citep[see, e.g.][]{dehaanferreira2006}.
For example, if $F$ corresponds to a generalized extreme value distribution with shape parameter  $\xi \in \mathbb R$, then 
\begin{equation*}\label{eq: generalized extreme value}
    \frac{3E (X F^2(X)  - E(X)}{2E (X F(X) ) - E(X)} = \frac{3^\xi-1}{2^\xi-1} ,
\end{equation*}
and a similar formula is available for the generalized Pareto distribution. 

Such moment equalities provide simple building blocks to quickly and efficiently  implement a method-of-moment to estimate both generalized extreme value distribution or generalized Pareto parameters. Two main approaches have been used to infer PWMs. The first one consists in replacing the function $F$ by its empirical version and taking the mean over the sample. The second one takes advantage of the link between PWMs and order statistics. More precisely, if $(X_1,\dots,X_m)$ is an independent and identically distributed (i.i.d.) sample with common distribution function $F$, and if $X_{(1:m)}\leq\dots \leq X_{(m:m)}$ is the ordered  sample, then a simple calculation shows that, for all $1\leq k\leq m$, 
\[
\theta_{k:m} = E(X_{(k:m)})= k{m\choose k}E\left(XF(X)^{k-1}\overline{F}(X)^{m-k}\right)\, .
\]
This indicates that, in the i.i.d. setting, the estimation of PWMs can be deduced from the order statistics. A natural choice for estimating $\theta_{k:m}$, with $1\leq k\leq m\leq n$, is thus to use the unbiased estimator
\begin{equation}\label{eq:est-alpha_s}
  \frac{1}{{n\choose m}}\sum_{1\leq i_1<\dots<i_{m}\leq n} \Psi_k(X_{i_1},\dots,X_{i_{m}})
\end{equation}
where $\Psi_k(X_{i_1},\dots,X_{i_{m}})$ corresponds to the $k$-th order statistic in the sub-sample $(X_{i_1},\dots,X_{i_{m}})$. 
For instance, \citet{Land:Mata:Wall:79} considered the special case of $k=1$. Those two approaches are closely related and their asymptotic properties have  been studied in detail \citep[see, e.g.][]{Hosk:Wall:Wood:85,Ferreira_2015,Dielbolt08,Guillou09, Diebolt03, Diebolt04, diebolt07}. 

The literature on non asymptotic properties of PWM estimators is, to our knowledge, sparse. 
\citet{Furrer07} derived explicit variance expressions for finite samples, but only in the case where the sample distribution is a generalized Pareto distribution. Estimators such as~\eqref{eq:est-alpha_s} have at least two drawbacks. First, in heavy-tailed scenarios where the underlying distribution has only low-order moments, estimate properties are not established for finite samples. In particular, classical concentration inequalities based on exponential decay of the tail cannot be directly applied to quantities like~\eqref{eq:est-alpha_s}. Second, they may be extremely sensitive to the presence of outliers in the sample. The main motivation of this work is thus to design estimators of $\theta_{k:m}$ with good concentration properties under a second moment assumption only, and that would be robust to the presence of outliers.

Let us mention that the treatment of outliers for PWM estimation has rarely been covered within the extreme value theory community~\citep[see, e.g.][]{Hubert08,Dupuis06}. 
Reducing  the  negative impact of  
outliers, i.e. large corrupted anomalies, on the estimation of   extreme value parameters 
demands a careful statistical analysis. Therefore, inference tools based on robust statistics \citep[see e.g.][]{Minsker20,Hubert08,lecue2019learning,devroye2016sub} need to be adapted to extreme value theory. For example,   
 \citet{Brazauskas00} leveraged the concept of  generalized quantiles 
 to obtain favourable trade-offs between efficiency  and robustness in the estimation of  the parameters of a generalized Pareto distribution. 
Recently, \citet{Bhattacharya19}  studied a  trimmed version of the Hill estimator to infer positive $\xi$ and they 
proposed a methodology to identify extreme outliers in heavy-tailed data. 
 \citet{bhattacharya2019outlier} extended their work to light tail distributions and built  a tail-adjusted boxplot. 
Still, all these studies focused on developing asymptotic distributions for their estimators, but non asymptotic bounds were not obtained. 

To derive concentration bounds without exponential moment assumption and to achieve robustness, we propose, in Section~\ref{sec:median-of-means}, to adapt the so-called median-of-means  concentration technique \citep[see e.g.][]{devroye2016sub,joly2016robust,lecue2019learning,lecue2020robustAoS} to the estimation of 
$\theta_{k:m}$. Our estimator is actually defined in the much more general context of estimating the mean of symmetric multivariate kernels, when usual $U$-statistics may not give reliable estimates. In Section~\ref{sec:main}, we establish non asymptotic performance bounds, for degenerate and non-degenerate kernels, with sharp variance proxys.  In addition, we show that our estimator is strongly robust to the presence of outliers in the sample, under a very generic contamination scheme introduced by~\citet{lecue2019learning}. Section~\ref{sec:bounds:xi} combines the problem of tail index estimation with our robust median-of-means inference scheme. In Section~\ref{sec:simulations}, numerical experiments are used to compare the classical PWM approach with our method. All proofs can be found in Section~\ref{sec:proofs}. In the appendix,  the bounds derived in Section~\ref{sec:main} are generalized beyond the i.i.d. setting by considering exchangeable sequences satisfying a negative dependence condition know as conditional negative association.

\section{Median-of-means estimators}
\label{sec:median-of-means}

In this section, we recall the  median-of-means techniques and the construction of the associated estimators. As its name suggests, a  median-of-means estimator is obtained as the median of means, the latter being, in this work, computed as $U$-statistics on independent blocks of the original given sample.

In the sequel, we assume that samples are all independent and identically distributed, unless otherwise specified. Let $X_1,\dots,X_n$ be a sample with values in some measurable set $\mathcal{X}$.  We are interested in the  robust estimation of quantities of the form 
\begin{equation}\label{def:theta}
\theta = E\left(\Psi(X_1,\dots,X_m)\right)\in\mathbb{R} 
\end{equation}
where, for an integer $m\geq 1$, $\Psi:\mathcal{X}^m\to \mathbb{R}$ is a symmetric function, called kernel. Let $v_k=\mathrm{var}\left(E\left[\Psi(X_1,\dots,X_m)\,|\,X_1,\dots,X_k\right]\right)$, the kernel $\Psi$ is said to be $q$-degenerate, for $q\in\{1,\dots,m\}$, if $v_1=\dots=v_{q-1}=0$ and $v_q>0$. If $v_1>0$,  $\Psi$ is said to be non-degenerate.

Assuming $n\geq m$, a natural estimator for $\theta$ is given by the following $U$-statistics
\[
U_n=\frac{1}{{n\choose m}}\sum_{\substack{A\subset [n]\\ |A|=m}} \Psi(X_A)\, ,
\]
where $[n]=\{1,\dots,n\}$ and $X_A=(X_i)_{i\in A}$. In the case of PWM, $\theta = \theta_{k:m} = E(X_{(k:m)})$ and $\Psi = \Psi_k$ corresponds to the $k$-th order statistics in a sample of size $m$.

As shown by \citet{joly2016robust}, robust estimation of the parameter $\theta$, defined in \eqref{def:theta}, can be obtained using median-of-means techniques. We present here a new estimator of $\theta$ based on these techniques. The main idea is to divide the sample $(X_1, \dots, X_n)$ into $K$ blocks, that is disjoint subsets.
To fix this number $K$ of blocks, the practitioner first needs to choose an error level $\delta \in  [\mathrm{e}^{-\lfloor n/m \rfloor},1[$ and then the integer   $K$  can be defined as  \(K=\left\lceil \log(1/\delta)\right\rceil\). 
By construction,   $1\leq K\leq \lfloor n/m \rfloor$ and the set  $[n]=\{1,\dots,n\}$ is then divided into $K$ disjoint blocks $B_1,\dots,B_K$, each of size $|B_j| \geq \left\lfloor n/K\right\rfloor\geq m$.

Within each block  $j$, we   construct the U-statistic estimator 
\[
\hat{\theta}^{(j)}=\frac{1}{{|B_j|\choose m}}\sum_{\substack{A\subset B_j\\ |A|=m}} \Psi(X_A)\, ,
\]
and then compute the median  among blocks, i.e. 
\begin{equation}\label{est:medianofmeans}
\hat{\theta}_n=\mathrm{median}\left( \hat{\theta}^{(1)},\dots, \hat{\theta}^{(K)}\right)\, ,
\end{equation}
where $\mathrm{median}(z_1,\dots,z_K)$ corresponds to the smallest value $z\in\{z_1,\dots,z_K\}$ such that
\[
\left|\left\{j\in [K]\,,\, z_j\leq z\right\}\right|\geq K/2\quad\text{ and }\quad \left|\left\{j\in [K]\,,\, z_j\geq z\right\}\right|\geq K/2\, .
\]

Obtaining non asymptotic concentration inequalities for $U_n$ with the correct variance factor turns out to be a difficult problem, even in the non-degenerate case. Hoeffding or Bernstein-type inequalities have been obtained, but they require the kernel $\Psi$ to be bounded or to have sufficiently light tails. For instance, if $|\Psi|\leq c$,  \citet{maurer2019bernstein} that, for all $t\geq 0$,
\[
\mathrm{pr}\left(|U_n-E(U_n)|\geq t\right)\leq 2\exp\left(-\frac{n t^2}{2m^2v_1+\frac{4m^4v_m}{n}+\frac{16m^2ct}{3}}\right)\, 
\] 
\citep[see also][]{arcones1995bernstein}. Previously, \citet{hoeffding1963} had shown a sub-Gaussian inequality for $U_n$. 

However, when $\Psi$ is heavy-tailed, exponential inequalities might not hold anymore. Since, PWM estimators are classically used to construct estimators in extreme value analysis. We are interested in deriving non asymptotic bounds with minimum moments conditions. In the next section, we show that our median-of-means estimator exhibits exponential concentration with the correct variance factor, see Proposition~\ref{prop:medians-degenerate}. We show also that this estimator is robust to the presence of outliers in the sample, in a very generic contamination scheme, see Proposition~\ref{prop:medians-outliers}.

\section{Sub-Gaussian and Bernstein-type bounds for median-of-means estimators}\label{sec:main}

In this section, we derive sub-Gaussian and Bernstein-type bounds for general median-of-means estimators with non-degenerate or degenerate kernel. Then, the issue of the robustness to the presence of outliers in the sample is also addressed. The results are stated in a general context and PWM estimators simply correspond to a particular case. 

\begin{proposition}\label{prop:medians-degenerate}
Let $X_1, \dots, X_n$ be an i.i.d. sample with values in $\mathcal{X}$, and, for a positive integer $m\in [n]$, let $\Psi:\mathcal{X}^m\to \mathbb{R}$ be a symmetric $q$-degenerate kernel, $q\in [m]$, with $E\left[\Psi(X_1,\dots,X_m)\right]=\theta$ and $v_m<\infty$. Then, for all $\delta \in [\mathrm{e}^{-\lfloor n/m \rfloor} ,1)$, 
the median-of means estimator $\hat{\theta}_n$ defined by \eqref{est:medianofmeans} with $K=\left\lceil \log(1/\delta)\right\rceil$ satisfies 
\begin{equation}\label{eq:prop-median-sub-gaussian-degenerate}
\mathrm{pr}\left(  | \hat{\theta}_n -\theta | > 2\mathrm{e}\sqrt{\frac{{m-1\choose q-1}(2m)^qv_m \left\lceil \log(1/\delta)\right\rceil^q}{n^q}}\right)\leq \delta\, ,
\end{equation}
and
\begin{equation}\label{eq:prop-median-sub-gamma-degenerate}
\mathrm{pr}\left(  | \hat{\theta}_n -\theta | > 2\mathrm{e}\sqrt{\frac{{m\choose q}(2m)^q v_q \left\lceil \log(1/\delta)\right\rceil^q}{n^q}+\frac{{m-1\choose q}(2m)^{q+1}v_m \left\lceil \log(1/\delta)\right\rceil^{q+1}}{n^{q+1}}}\right)\leq \delta\, .
\end{equation}
\end{proposition}

Bound~\eqref{eq:prop-median-sub-gaussian-degenerate} is  similar to the one obtained by~\citet{joly2016robust} for their estimator. Bound~\eqref{eq:prop-median-sub-gamma-degenerate} is new and  gives a significant improvement over~\eqref{eq:prop-median-sub-gaussian-degenerate}, especially in the regime where $\log(1/\delta)/n=o(1)$. In that case, the second term under the square root can be neglected, and the variance factor in the first term is improved to
\[
\frac{{m\choose q} (2m)^qv_q}{n^q} \, , 
\]
which is close to the asymptotic variance obtained by \citet{hoeffding1948} which states that
\begin{equation*}\label{eq:asymp-variance}
\mathrm{var}\left(U_n\right) \underset{n\to\infty}{\sim} q!{m\choose q}^2 \frac{v_q}{n^q}\, ,
\end{equation*}
for non-degenerate kernels, $\mathrm{var}(U_n)\underset{n\to\infty}{\sim}m^2v_1/n$, and $U_n$ is asymptotically normal:
\[
\sqrt{n}(U_n-\theta)\underset{n\to\infty}{\rightsquigarrow} \mathcal{N}(0,m^2v_1)\, .
\]
In the case $q=1$, i.e.\ in the non-degenerate case, then Equation~\eqref{eq:prop-median-sub-gaussian-degenerate} states that $\hat{\theta}_n -\theta$ is sub-Gaussian on both tails with variance factor proportional to $mv_m/n$:
\begin{equation}\label{eq:prop-median-sub-gaussian}
\mathrm{pr}\left(  | \hat{\theta}_n -\theta | > 2\mathrm{e}\sqrt{\frac{2mv_m \left\lceil \log(1/\delta)\right\rceil}{n}}\right)\leq \delta\, ,
\end{equation}
while Equation~\eqref{eq:prop-median-sub-gamma-degenerate} states that it is sub-gamma on both tails with variance factor proportional to $m^2v_1/n$ and scale factor proportional to $\sqrt{m^3v_m}/n$:
\begin{equation}\label{eq:prop-median-sub-gamma}
\mathrm{pr}\left(  | \hat{\theta}_n -\theta | > 2\mathrm{e}\sqrt{\frac{2m^2v_1 \left\lceil \log(1/\delta)\right\rceil}{n}+\frac{4m^3v_m \left\lceil \log(1/\delta)\right\rceil^2}{n^2}}\right)\leq \delta\, .
\end{equation}
Since $mv_1\leq v_m$, the variance factor in~\eqref{eq:prop-median-sub-gamma} is always smaller than the one in~\eqref{eq:prop-median-sub-gaussian}. Due to the scale factor in~\eqref{eq:prop-median-sub-gamma}, either inequality might be better, depending on the value of $\delta$. 

\begin{remark}\label{prop:medians-independent}
The estimator $\hat{\theta}_n$  depends on the pre-chosen  confidence threshold $\delta$.  As shown by~\citet[Theorem 4.2]{devroye2016sub}, when an upper-bound $v^\star_{m}$ on $v_{m}$ is available, an estimator $\tilde{\theta}_n$ independent of $\delta$ may be constructed.
The notation $\hat{\theta}_{n,\delta}$ below highlights this connection.  From Proposition~\ref{prop:medians-degenerate}, for all $k\in \{1,\dots,\lfloor n/m\rfloor\}$, the interval
\[
\hat{I}_k=\left[ \hat{\theta}_{n,\mathrm{e}^{-k}} \pm 2\mathrm{e}\sqrt{\frac{2mv^\star_{m}k}{n}}\right]
\]
is a confidence interval with level $1-\mathrm{e}^{-k}$, where $\hat{\theta}_{n,\mathrm{e}^{-k}}$ denote the estimator~\eqref{est:medianofmeans} obtained with $\delta =\mathrm{e}^{-k} $. Now, let
\[
\hat{k}=\min\left\{ 1\leq k\leq \lfloor n/m\rfloor\, , \bigcap_{j=k}^{\lfloor n/m\rfloor} \hat{I}_j\neq \emptyset\right\}\, ,
\]
and define the estimator $\tilde{\theta}_n$ as the midpoint of the interval $\bigcap_{j=\hat{k}}^{\lfloor n/m\rfloor}\hat{I}_j$. Then, the estimator $\tilde{\theta}_n$ satisfies, for all $\delta\in \left[ \mathrm{e}^{-\lfloor n/m\rfloor}/(1-\mathrm{e}^{-1}),1\right[$,
\begin{equation}\label{eq:medians-independent}
\mathrm{pr}\left( | \tilde{\theta}_n -\theta |>  4\mathrm{e}\sqrt{\frac{2mv^\star_{m}\left(1+\log\left(\tfrac{1}{1-\mathrm{e}^{-1}}\right)+\log\left(\frac{1}{\delta}\right)\right)}{n}}\right)\leq \delta\, .
\end{equation}

\end{remark}

Recall that our main goal was two-fold. First, we aim at proposing a family of estimators of $\theta$ for which sharp concentration bounds are available under minimal moment conditions. Then, we also wish to address the issue of robustness of these estimators. This is the purpose of the next result which concerns the robustness of the estimator $\hat{\theta}_n$ to the presence of outliers in the original sample. To translate the presence of outliers, we consider the contamination scheme introduced by~\citet{lecue2019learning}: the index set $[n]$ is divided into two disjoint subsets, the subset $\mathcal{I}$ of inliers, and the subset $\mathcal{O}$ of outliers. The sequence $(X_i)_{i\in\mathcal{I}}$ is i.i.d. while no assumption is made on the variables $(X_i)_{i\in\mathcal{O}}$. In what follows, $\mathrm{pr}_{\mathcal{O}\cup\mathcal{I}}$ corresponds to the distribution of such a contaminated sample.

\begin{proposition}\label{prop:medians-outliers}
Let $(X_1, \dots, X_n)$ be a contaminated sample under the $\mathcal{O}\cup\mathcal{I}$ model. For all $\delta \in [\mathrm{e}^{-\lfloor n/m \rfloor} ,1)$, let $\hat{\theta}_n$ be defined as in~\eqref{est:medianofmeans} with $K=\left\lceil \log(1/\delta)\right\rceil$. If $|\mathcal{O}|\leq K/4$, then 
\[
\mathrm{pr}_{\mathcal{O}\cup\mathcal{I}}\left( | \hat{\theta}_n -\theta | > \frac{16\mathrm{e}^2}{3\sqrt{3}}\sqrt{\frac{2mv_m \left\lceil \log(1/\delta)\right\rceil}{n}}\right)\leq \delta\, ,
\]
and
\[
\mathrm{pr}_{\mathcal{O}\cup\mathcal{I}}\left( | \hat{\theta}_n -\theta | > \frac{16\mathrm{e}^2}{3\sqrt{3}}\sqrt{\frac{2m^2v_1 \left\lceil \log(1/\delta)\right\rceil}{n}+\frac{4m^3v_m \left\lceil \log(1/\delta)\right\rceil^2}{n^2}}\right)\leq \delta\, .
\]
\end{proposition}

To obtain non asymptotic bounds for $\theta_{k:m}$, it suffices to take $\Psi=\Psi_k$, that is the $k$-th order statistic in a sample of size $m$, which is symmetric and non-degenerate. This introduces our new class of robust estimators  $\hat{\theta}_{k:m}$ of $\theta_{k:m}=E(X_{(k:m)})$ as defined in~\eqref{est:medianofmeans} with $K=\left\lceil\log(1/\delta)\right\rceil$ for some $\delta\in [\mathrm{e}^{-\lfloor n/m\rfloor}, 1[ $. Thus, $\hat \theta_{k:m}$ satisfies sub-Gaussian and Bernstein-type inequalities for uncontaminated or contaminated samples, that is $\hat \theta_{k:m}$ satisfies Propositions~\ref{prop:medians-degenerate} and~\ref{prop:medians-outliers} with $v_m = \mathrm{var}(X_{(k:m)})$ and $v_1 = \mathrm{var}(E(X_{(k:m)}|X_1))$.

\section{Sub-Gaussian and Bernstein-type bounds for the median-of-means PWM estimator of the tail index}\label{sec:bounds:xi}

Considering a sample $X_1,\ldots, X_n$ distributed according to a generalized extreme value distribution  $G_{\xi}$. As mentioned in the introduction, the parameter $\xi$ can be linked to PWM. Here, we consider the following explicit expression proposed by \citet[][Remark 2.2]{Ferreira_2015}:
\begin{equation}\label{eq:xi:gev}
\xi = \frac{1}{\log 2}\log \left(\frac{4E(XG_\xi^3(X)) - 2 E(XG_\xi(X)) }{2E(XG_\xi(X)) -  E(X) }\right)=\frac{1}{\log 2}\log \left(\frac{\theta_{4} - \theta_{2} }{\theta_{2} -  \theta_{1} }\right)\, ,
\end{equation}
where $\theta_j=\theta_{j:j}=E(\max(X_1,\dots,X_j))$. Therefore, a natural estimator of $\xi$ is obtained by simply estimating the $\theta_j$, $j=1,2,4$ in~\eqref{eq:xi:gev}  and a median-of-means PWM estimator of $\xi$ is thus given by
\begin{equation}\label{eq:est:xi}
\hat \xi_n = \frac{1}{\log 2}\log \left(\frac{\hat\theta_{4} - \hat\theta_{2} }{\hat\theta_{2} - \hat\theta_{1} }\right) \, , 
\end{equation}
where $\hat{\theta}_j$ be the median-of-means estimate $\theta_{j}$ for $j\in\{1,2,4\}$ constructed with $K=\left\lceil \log(1/\delta)\right\rceil$ blocks for $\delta\in [\mathrm{e}^{-\lfloor n/4 \rfloor},1[$. The next result shows that $\xi_n$ is sub-Gaussian on both tails for uncontaminated or contaminated samples.

\begin{proposition}\label{prop:ineq-xi}
Let $X_1,\ldots, X_n$ be a sample with distribution function $F$. Then, 
 \[
\mathrm{pr} \left(|\hat{\xi}_n-\xi| \geq \frac{2\mathrm{e}(2^\xi +1)(2^{-\hat{\xi}}+2^{-\xi})}{\log(2)(\hat{\theta}_{2}-\hat{\theta}_{1})}\sqrt{\frac{8\max(v_{1},v_{2},v_{4})\left\lceil \log(1/\delta)\right\rceil}{n}} \right) \leq 3\delta\, ,
 \]
 where $v_j=\mathrm{var}\left(\max(X_1,\dots,X_j)\right)$.

 Moreover, if the sample $X_1,\ldots, X_n$ is contaminated by outliers with the same scheme of Proposition~\ref{prop:medians-outliers}, then 
 \[
 \mathrm{pr}_{\mathcal{O} \cup \mathcal{I}} \left(|\hat{\xi}_n-\xi| \geq \frac{16\mathrm{e}^2(2^\xi +1)(2^{-\hat{\xi}}+2^{-\xi})}{3\log(2) \sqrt 3(\hat{\theta}_{2}-\hat{\theta}_{1})}\sqrt{\frac{8\max(v_{1},v_{2},v_{4})\left\lceil \log(1/\delta)\right\rceil}{n}}\right) \leq 3\delta\, .
 \]
 \end{proposition}

In extreme value theory, statisticians are interested in the estimation of this parameter $\xi$ which reflects the heaviness of the tail of the distribution. Our median-of-mean estimator thus provide a robust and with good finite sample properties estimation of such quantity.

\section{Numerical experiments}\label{sec:simulations}

This section illustrates the robustness of our median-of-means PWM estimators $\hat \theta_{k:m}$ and $\hat \xi$. We consider the following setting. We simulate 1~000 generalized extreme value samples contaminated by outliers. For that purpose, we first draw samples of inliers of size $n_{\mathcal I}$ distributed according a generalized extreme value distribution with 3 different values for the tail index $\xi=$-0.4, 0 and 0.4. These samples are then contaminated by $n_{\mathcal O}$ outliers, such that $n=n_{\mathcal I}+n_{\mathcal O}$ constant equal to 200. To simulate the outliers, we distinguish two cases:  $\xi <0$ and $\xi \geq 0$. When $\xi <0$, the probability that a generalized extreme value variable with parameter $\xi<0$ exceeds any value $-1/\xi$ is equal to 0. The outliers $X_i$, $i=1,\ldots, n_{\mathcal O}$ are then obtained as 
\[
X_i \sim \text{Unif}[0,20-\frac{1}{\xi}]  , \quad i=1,\ldots,n_{\mathcal O}. 
\] When $\xi  \geq 0$, The outliers $X_i$, $i=1,\ldots, n_{\mathcal O}$ are obtained 
\[
X_i \sim \mathcal{N}(q_\xi(1-10^{-4}), 1)  , \quad i=1,\ldots,n_{\mathcal O}. 
\]
where  $q_\xi(1-10^{-4})$ denotes the $(1-10^{-4}$)-quantile of a generalized extreme value of parameter~$\xi$. 

\begin{figure}
    \centering
    \begin{tabular}{c}
    \begin{tabular}{cc}
    \includegraphics[width= 0.45\linewidth]{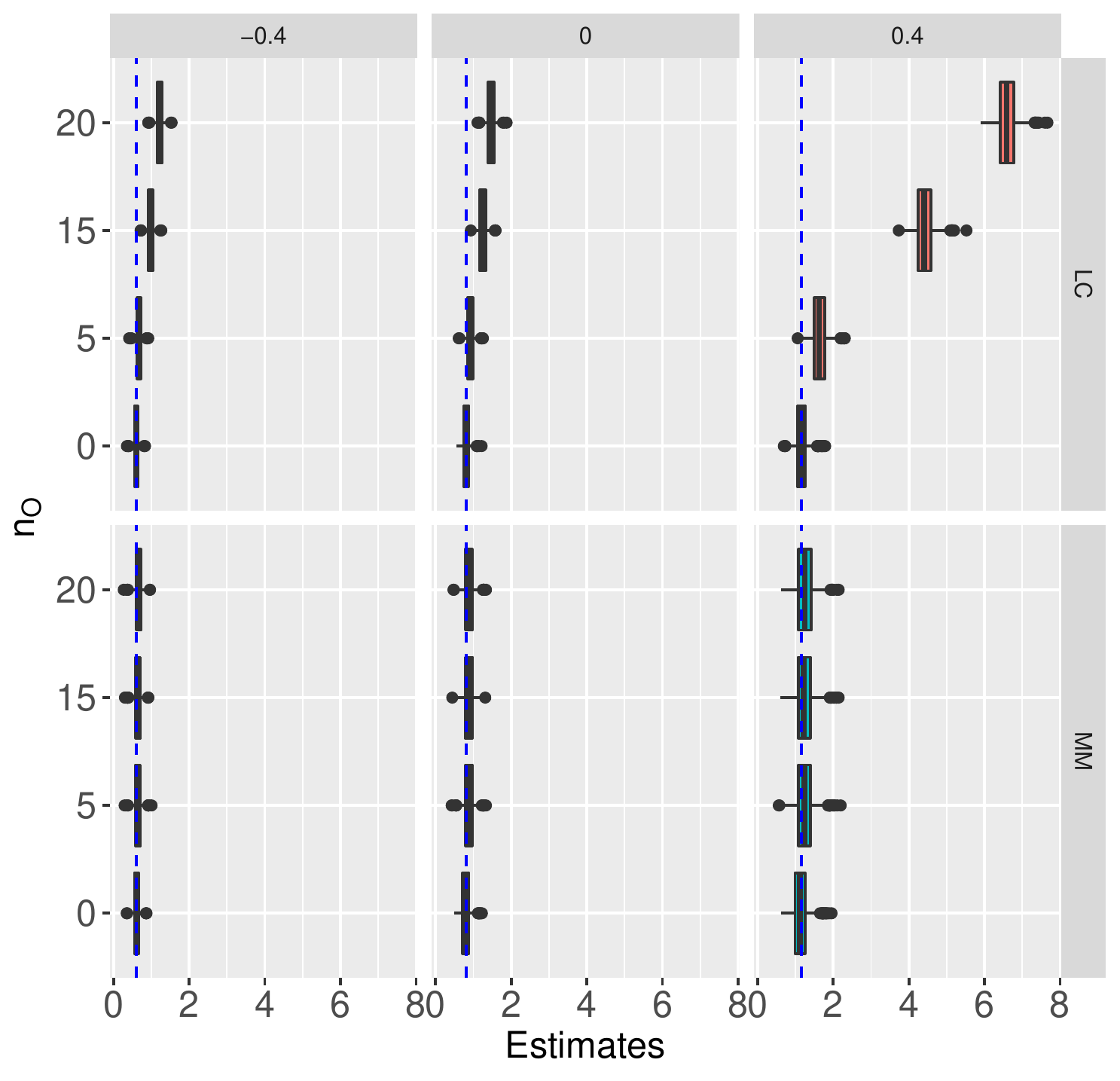}
 &   \includegraphics[width= 0.45\linewidth]{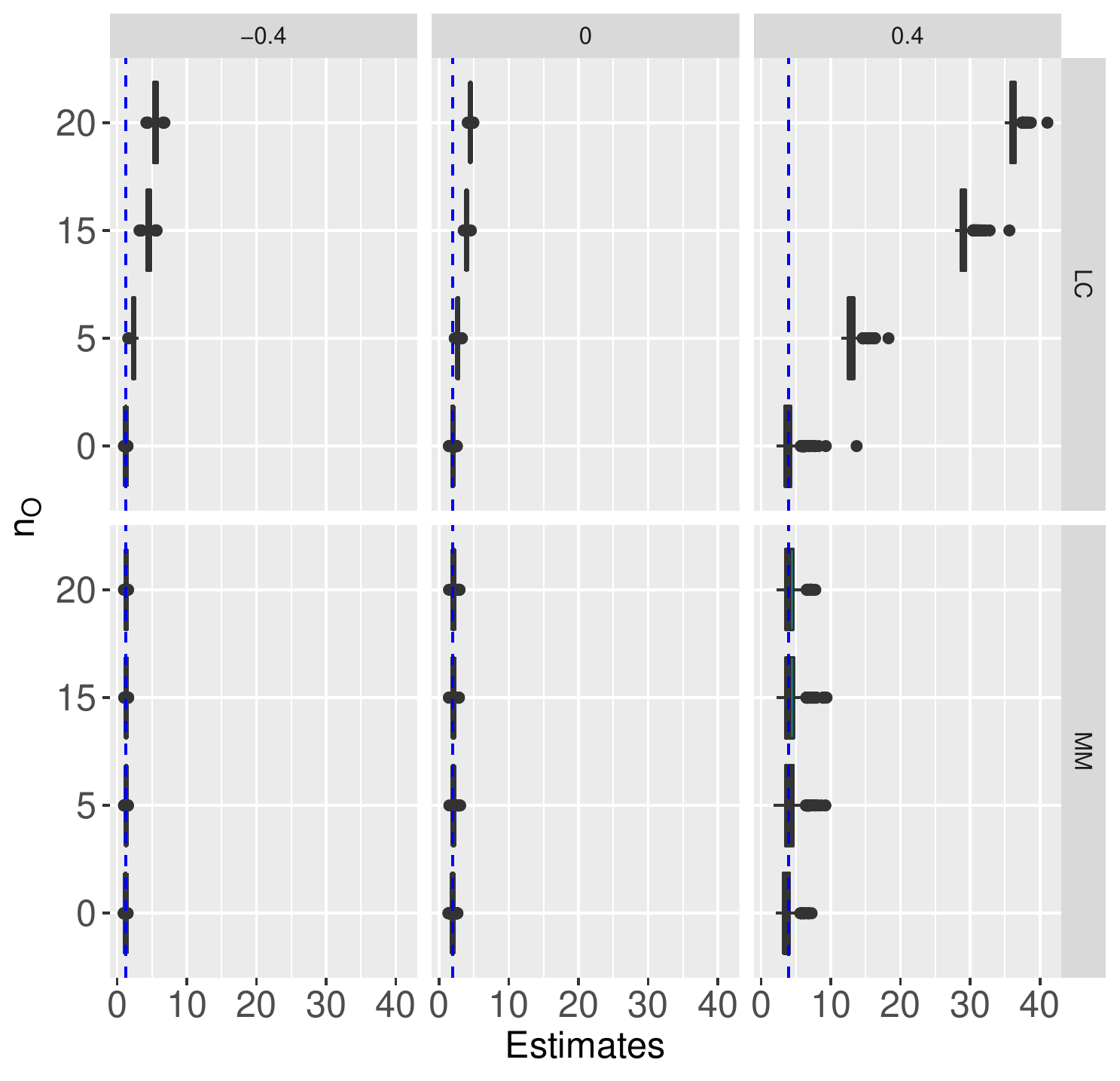} \\
   a) & b) 
    \end{tabular}\\~\\
    \begin{tabular}{cc}
    \includegraphics[width= 0.45 \linewidth]{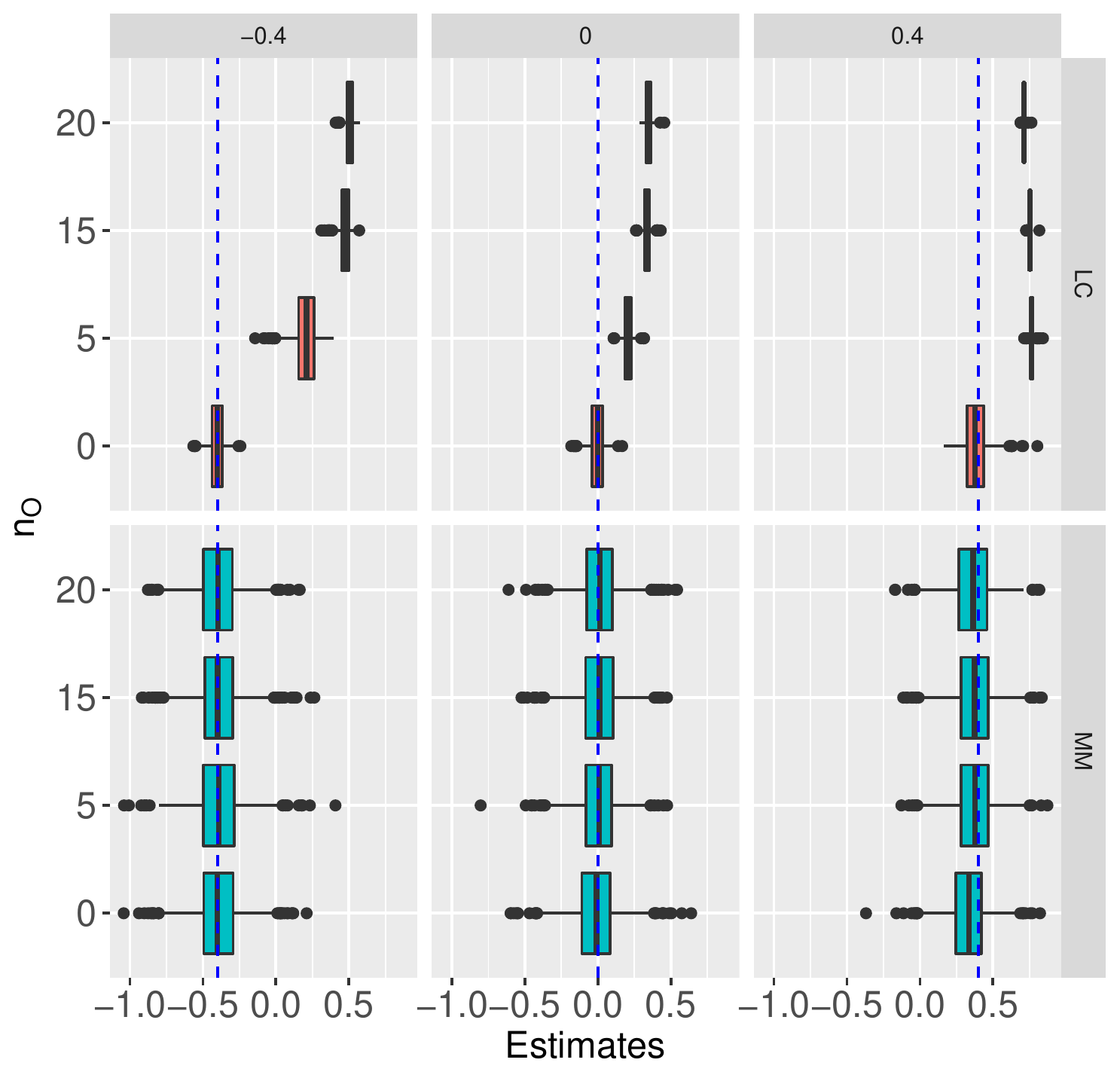} &   \includegraphics[width= 0.45\linewidth]{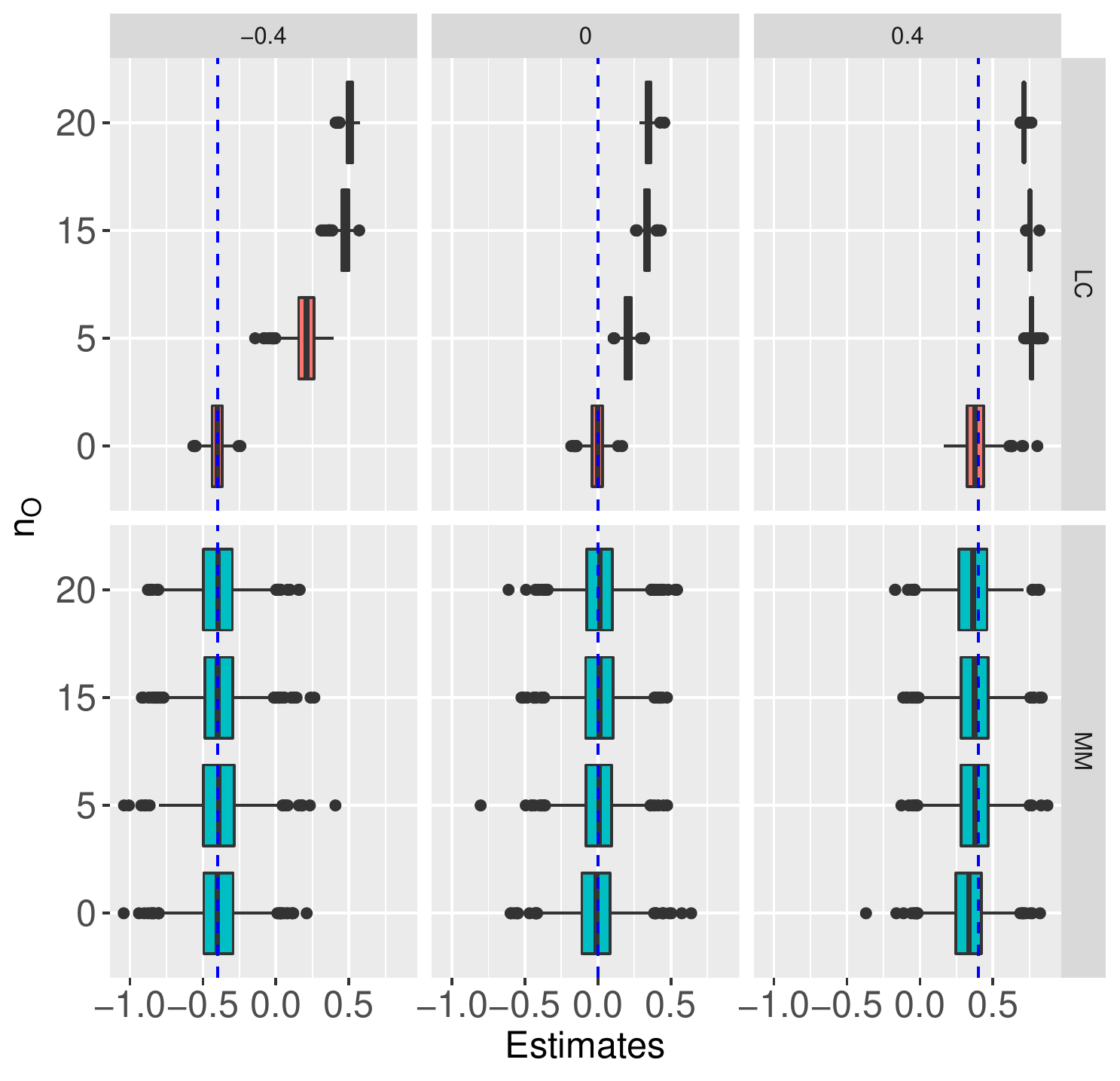} \\
   c) & d) 
   \end{tabular}
    \end{tabular}
    \caption{Estimates of a) $\theta_{3:4}$, b)  $\theta_{4:4}$,  c) $\xi$ and d) $q_{\xi}(0.95)$ obtained with the median-of-mean estimator (MM, in blue, see Equation \eqref{eq:est:xi}), and the estimator computed as a linear combination (LC, in red, see Equation \eqref{eq:comb:linear}), when the shape parameter $\xi$ is equal to -0.4, 0 and 0.4, in columns, and the number of outliers $n_{\mathcal O}$ is equal to 0, 5, 15 and 20, in rows. In each graph, the blue dotted line corresponds to the true value of the quantity to be estimated.}
    \label{fig:boxplot}
\end{figure}

 Let us highlight that when the kernel $\Psi=\Psi_k$, the procedure to build $\hat{\theta}_{k:m}$ 
does not require to solve an optimisation  problem (in contrast to maximum likelihood techniques) and is computationally straightforward to implement. In this special case, computing the $U$-statistics $\hat{\theta}^{(j)}$ within bloc $B_j$ does not require going through all the $m$ subsets of $B_j$. Instead, we may notice that
\begin{equation}\label{eq:comb:linear}
\frac{1}{{n \choose m}}\sum_{\substack{A\subset [n]\\ |A|=m}} \Psi_k\left(X_A\right)=\frac{1}{{n \choose m}}\sum_{i=1}^n { n-i\choose m-k}{i-1 \choose k-1}X_{(i:n)} := T_{k:m}\, .
\end{equation}

For comparison, we also consider the estimator $T_{k:m}$ of $\theta_{k:m}$ which corresponds to the classical $U$-statistics estimate over the whole sample (taking $K=1$ block). As a weighted mean of the order statistics, it is very sensitive to the presence of outliers.

In both cases, we estimate $\theta_{k:m}$, for $m=1,\ldots,4$ and $k \leq m$, the parameter $\xi$ and the 0.95-quantile, denoted $q_{\xi}(0.95)$, with the linear combination estimator $T_{k:m}$ and the median-of-means estimator $\hat \theta_{k:m}$. Figure~\ref{fig:boxplot} displays the boxplots of both estimators as the number of outliers $n_{\mathcal O}$ varies for $\theta_{3:4}$, $\theta_{4:4}$, $\xi$ and $q_{\xi}(0.95)$ (the other graphs can be found along with the code online \texttt{git@github.com:maudmhthomas/PWM.git}). It can be seen that the classical estimator is highly sensitive to outliers. Our median-of-means estimator inherits the robustness properties of the median.

\section{Proofs}\label{sec:proofs}

To prove Proposition~\ref{prop:medians-degenerate}, we first establish the following lemma, which gives non asymptotic bounds on the variance of $U$-statistics.

\begin{lemma}\label{lemma:variance-block}
Let $\Psi:\mathcal{X}^m\to\mathbb{R}$ be a symmetric $q$-degenerate kernel, with $q\in [m]$, such that $v_m<\infty$. For $n\geq m$, let $X_1,\dots,X_n$ be an i.i.d. sample, and define
\[
U_n=\frac{1}{{n \choose m}}\sum_{\substack{A\subset [n]\\ |A|=m}} \Psi\left(X_A\right)\, .
\]
Then, 
\begin{equation}\label{eq:lemma-variance-1}
\mathrm{var}( U_n)\leq \frac{{m-1\choose q-1}m^qv_m}{n^q}\, ,
\end{equation}
and
\begin{equation}\label{eq:lemma-variance-2}
\mathrm{var}( U_n)\leq \frac{{m\choose q}m^q v_q}{n^q} +\frac{{m-1\choose q}m^{q+1}v_m}{n^{q+1}}\, .
\end{equation}
\end{lemma}

\begin{proof}[of Lemma~\ref{lemma:variance-block}]
Recall the identity
\begin{equation}\label{eq:bound-variance-u-stat}
 \mathrm{var}\left(U_n\right)= \frac{1}{{n\choose m}}\sum_{k=1}^m {m\choose k}{n-m\choose m-k} v_k\, .
\end{equation}
Now if $\Psi$ is $q$-degenerate, then $v_1=\dots=v_{q-1}=0$, and the sum above may be started at $k=q$.  \citet[Theorem 5.1]{hoeffding1948} showed that, for all $1\leq k\leq\ell \leq m$,
\begin{equation}\label{eq:ineq-variance-hoeffding}
\frac{v_k}{k}\leq \frac{v_\ell}{\ell}\, .
\end{equation}
Hence, noting that $v_k\leq kv_m/m$,
\begin{align*}
     \mathrm{var}\left(U_n\right)&\leq \frac{1}{{n\choose m}}\sum_{k=q}^m {m\choose k}{n-m\choose m-k} \frac{k}{m}v_m\\
     &=\frac{v_m}{{n\choose m}}\sum_{k=q-1}^{m-1}{m-1\choose k}{n-m\choose m-1-k}\\
     &\leq \frac{v_m{m-1\choose q-1}{n-q\choose m-q}}{{n\choose m}}\, ,
\end{align*}
where the last inequality comes from the fact that the number of ways to choose $m-1$ elements in a set of size $n-1$ with at least $q-1$ elements taken from a given subset of size $m-1$ is less that the number of ways to first choose $q-1$ elements in that subset, and then $m-q$ elements in what remains. Then, since
\[
\frac{{n-q\choose m-q}}{{n\choose m}}\leq \left(\frac{m}{n}\right)^q \, , 
\]
we obtain
\[
\mathrm{var}\left(U_n\right)\leq \frac{{m-1\choose q-1}m^qv_m}{n^q}\, ,
\]
establishing~\eqref{eq:lemma-variance-1}. The second bound~\eqref{eq:lemma-variance-2} is obtained by singling out the term corresponding to $k=q$ in~\eqref{eq:bound-variance-u-stat} and using $v_k\leq kv_m/m$ only for $k\geq q+1$. This yields
\[
    \mathrm{var}\left(U_n\right)\leq \frac{{m\choose q}{{n-m}\choose {m-q}}}{{n\choose m}}v_q +\frac{1}{{n\choose m}}\sum_{k=q+1}^m {m\choose k}{n-m\choose m-k} \frac{k}{m}v_m\, .
 \]
    Now, on the one hand, 
\[
    \frac{{m\choose q}{n-m\choose m-q}}{{n\choose m}}\leq {m\choose q}\frac{{n-q\choose m-q}}{{n\choose m}} \leq {m\choose q}\left(\frac{m}{n}\right)^q \, \cdot 
\]
    On the other hand,
    \begin{align*}
    \frac{1}{{n\choose m}}\sum_{k=q+1}^m {m\choose k}{n-m\choose m-k} \frac{k}{m} &= \frac{1}{{n\choose m}}\sum_{k=q+1}^m {m-1\choose k-1}{n-m\choose m-k}\\
    &= \frac{1}{{n\choose m}}\sum_{k=q}^{m-1} {m-1\choose k}{n-m\choose m-1-k}\\
    &\leq \frac{{m-1\choose q}{n-q-1\choose m-q-1}}{{n\choose m}}\\
    &\leq  {m-1\choose q}\left(\frac{m}{n}\right)^{q+1} \, \cdot 
    \end{align*}
Finally,
    \[
    \mathrm{var}\left(U_n\right)\leq \frac{{m\choose q}m^q v_q}{n^q} +\frac{{m-1\choose q}m^{q+1}v_m}{n^{q+1}}\, 
    \]
    establishing~\eqref{eq:lemma-variance-2}.
\end{proof}

Before proving Proposition~\ref{prop:medians-degenerate}, we first state a well-known but useful fact. We include the proof for completeness. 

\begin{lemma}\label{lemma:laplace}
    Let $(Y_1,\dots,Y_K)$ be independent Bernoulli random variables such that for all $j\in [K]$, we have $E(Y_j)\leq p$, for some $p\in (0,1)$. Then, for all $a\in (0,1)$, 
    \[
    \mathrm{pr}\left(\sum_{j=1}^K Y_j\geq aK\right)\leq \left(\frac{p}{a}\right)^{aK}\left(\frac{1-p}{1-a}\right)^{(1-a)K}\, .
    \]
    In particular, for $a=1/2$, 
     \[
    \mathrm{pr}\left(\sum_{j=1}^K Y_j\geq K/2\right)\leq \left(2\sqrt{p(1-p)}\right)^{K}\leq \left(2\sqrt{p}\right)^{K}\, .
    \]
\end{lemma}

\begin{proof}[of Lemma~\ref{lemma:laplace}]
Let $\lambda\geq 0$, then 
\[
\log E\left( \mathrm{e}^{\lambda \sum_{j=1}^K Y_j}\right)=\sum_{j=1}^K \log E\left( \mathrm{e}^{\lambda Y_j}\right) =\sum_{j=1}^K \log\left( E(Y_j)\mathrm{e}^\lambda +1-E(Y_j)\right)\, .
\]
Since $E(Y_j)\leq p$, 
\[
\log E\left( \mathrm{e}^{\lambda \sum_{j=1}^K Y_j}\right)\leq K\log\left( p\mathrm{e}^\lambda +1-p\right)\, .
\]
Next, using a Chernoff bound, 
\[
    \mathrm{pr}\left( \sum_{j=1}^K Y_j\geq a K\right)\leq \exp \left\{-K\sup_{\lambda\geq 0}\left(\lambda a-\log(p\mathrm{e}^\lambda +1-p)\right)\right\}\, .
\]
Observing that the supremum is attained for $\lambda=\log\left(a(1-p)/(p(1-a))\right)$, 
\[
 \mathrm{pr}\left( \sum_{j=1}^K Y_j\geq a K\right)\leq \left(\frac{p}{a}\right)^{aK}\left(\frac{1-p}{1-a}\right)^{(1-a)K}\, .
\]
\end{proof}

\begin{proof}[of Proposition~\ref{prop:medians-degenerate}]
Let $t=\min\{t_1,t_2\}$, with 
\begin{equation}\label{eq:t1}
t_1=2e\sqrt{\frac{{m-1\choose q-1}(2m)^qv_m \left\lceil\log(1/\delta)\right\rceil^q}{n^q}}\, ,
\end{equation}
and
\begin{equation}\label{eq:t2}
t_2=2e\sqrt{\frac{{m\choose q}(2m)^q v_q \left\lfloor\log(1/\delta)\right\rfloor^q}{n^q}+\frac{{m-1\choose q}(2m)^{q+1}v_m \left\lceil\log(1/\delta)\right\rceil^{q+1}}{n^{q+1}}}\, .
\end{equation}
By definition of the median, both the number of $j$ such that $\hat{\theta}^{(j)}\geq \hat{\theta}_{n}$ and the number of $j$ such that $\hat{\theta}^{(j)}\leq \hat{\theta}_{n}$ are at least $K/2$. This leads to write
\begin{align*}
\mathrm{pr}\left(  |\hat{\theta}_{n} -\theta| >t\right) &\leq \mathrm{pr}\left( \sum_{j=1}^K Y_j \geq K/2\right)\, ,
\end{align*}
where $Y_j=\mathbf{1}_{\{  |\hat{\theta}^{(j)} -\theta| >t \}}$. We now look for an upper bound on $E(Y_j)$ so as to apply Lemma~\ref{lemma:laplace}. By Chebyshev Inequality,  for all $j\in\{1,\dots,m\}$,
\[
E(Y_j)=\mathrm{pr}\left(  |\hat{\theta}^{(j)} -\theta |>t\right)\leq  \frac{\mathrm{var}\left(\hat{\theta}^{(j)}\right)}{t^2}\, \cdot 
\]
When $t=t_1$, one may use the bound~\eqref{eq:lemma-variance-1} in Lemma~\ref{lemma:variance-block} with $|B_j|$ instead of $n$. Since $|B_j|\geq \left\lfloor n/K \right\rfloor\geq n/(2K)$, 
\[
E(Y_j)\leq \frac{{m-1\choose q-1}m^qv_{m}}{|B_j|^q t_1^2}\leq \frac{{m-1\choose q-1}(2mK)^qv_{m}}{n^q t_1^2}= \frac{1}{4\mathrm{e}^2}\, .
\]
When $t=t_2$, the bound~\eqref{eq:lemma-variance-2} in Lemma~\ref{lemma:variance-block} entails
\[
E(Y_j)\leq \frac{\frac{{m\choose q}m^qv_q}{|B_j|^q}+\frac{{m-1\choose q}m^{q+1}v_m}{|B_j|^{q+1}}}{t_2^2}\leq \frac{\frac{{m\choose q}(2mK)^qv_q}{n^q}+\frac{{m-1\choose q}(2mK)^{q+1}v_m}{n^{q+1}}}{t_2^2}=\frac{1}{4\mathrm{e}^2}\, ,
\]
where as above, we used that $|B_j|\geq \left\lfloor n/K \right\rfloor\geq n/(2K)$. In both cases, Lemma~\ref{lemma:laplace} can be applied with $a=1/2$ and $p=\mathrm{e}^{-2}/4$ to obtain
\[
\mathrm{pr}\left( \sum_{j=1}^K Y_j \geq K/2\right)\leq \mathrm{e}^{-K}\leq \delta\, .
\]
\end{proof}

Concerning the proof of Equation \eqref{eq:medians-independent} in Remark~\ref{prop:medians-independent}. Let $\delta\in \left[\mathrm{e}^{-\lfloor n/m\rfloor}/(1-\mathrm{e}^{-1}),1\right[$ and let $k_\delta$ be the smallest integer $k\in \{1,\dots, \lfloor n/m\rfloor\}$ such that $\delta\geq \mathrm{e}^{-k}/(1-\mathrm{e}^{-1})$. By a union bound, 
\begin{align*}
\mathrm{pr}\left( \bigcap_{j=k_\delta}^{\lfloor n/m\rfloor} \left\{ \theta \in \hat{I}_j\right\}\right) & \geq 1-\sum_{j=k_\delta}^{\lfloor n/m\rfloor} \mathrm{e}^{-j} \geq 1- \frac{\mathrm{e}^{-k_\delta}}{1-\mathrm{e}^{-1}} \geq 1-\delta\, .
\end{align*}
Now, on the event $\bigcap_{j=k_\delta}^{\lfloor n/m\rfloor} \left\{ \theta\in \hat{I}_j\right\}$, one has $\bigcap_{j=k_\delta}^{\lfloor n/m\rfloor} \hat{I}_j\neq \emptyset$, hence $\hat{k}\leq k_\delta$. But if $\hat{k}\leq k_\delta$, then $\tilde{\theta}_n$ also belongs to $\bigcap_{j=k_\delta}^{\lfloor n/m\rfloor}\hat{I}_j$ and 
\[
| \tilde{\theta}_n -\theta |\leq 4e\sqrt{\frac{2mv^\star_{m}k_\delta}{n}}\leq 4e\sqrt{\frac{2mv^\star_{m}\left(1+\log\left(\frac{1}{1-\mathrm{e}^{-1}}\right)+\log\left(\frac{1}{\delta}\right)\right)}{n}}\, ,
\]
which proofs Equation \eqref{eq:medians-independent}.

\begin{proof}[of Proposition~\ref{prop:medians-outliers}]
Let $t=8\mathrm{e}\min\{t_1,t_2\}/(3\sqrt{3})$, with $t_1$ and $t_2$ as defined in~\eqref{eq:t1} and~\eqref{eq:t2}. Then,
\[
\mathrm{pr}_{\mathcal{O}\cup\mathcal{I}}\left( |\hat{\theta}_n-\theta |> t\right)\leq \mathrm{pr}_{\mathcal{O}\cup\mathcal{I}}\left(\sum_{j=1}^K Y_j \geq K/2\right)\, ,
\]
with $Y_j=\mathbf{1}_{\left\{ \left|\hat{\theta}^{(j)}-\theta\right|> t\right\}}$. Letting $\mathcal{B}$ be the set of blocks that do not intersect $\mathcal{O}$, 
\[
\mathrm{pr}_{\mathcal{O}\cup\mathcal{I}}\left(\sum_{j=1}^K Y_j \geq K/2\right)\leq \mathrm{pr}_{\mathcal{O}\cup\mathcal{I}}\left(|\mathcal{B}^c|+\sum_{j\in\mathcal{B}} Y_j \geq K/2\right)\, .
\]
Since $|\mathcal{B}^c|\leq |\mathcal{O}|\leq K/4$ by assumption, we get
\[
\mathrm{pr}_{\mathcal{O}\cup\mathcal{I}}\left( |\hat{\theta}_n-\theta |> t\right)\leq  \mathrm{pr}\left(\sum_{j=1}^K Y_j\geq K/4\right)\, ,
\] 
 where $\mathrm{pr}$, here, corresponds to the probability measure for an i.i.d. sample. By Chebyshev Inequality, we have
\[
E(Y_j)\leq \frac{\mathrm{var}\left(\hat{\theta}^{(j)}\right)}{t^2}\, \cdot 
\]
In the proof of Proposition~\ref{prop:medians-degenerate}, we have shown that
\[
\frac{\mathrm{var}\left(\hat{\theta}^{(j)}\right)}{\min\{t_1,t_2\}^2}\leq \frac{1}{4\mathrm{e}^2}\, ,
\]
leading to $E(Y_j)\leq 3^3 2^{-8}\mathrm{e}^{-4}$.
We may thus apply Lemma~\ref{lemma:laplace} with $a=1/4$ and $p=3^3 2^{-8}\mathrm{e}^{-4}$ to obtain
\[
\mathrm{pr}\left(\sum_{j=1}^K Y_j \geq K/4\right) \leq \mathrm{e}^{-K}\leq \delta\, .
\]
\end{proof}

\begin{proof}[of Proposition~\ref{prop:ineq-xi}]
 Let 
 \[
c_n(\delta)=\frak c \sqrt{\frac{8\max(v_{1},v_{2},v_{4})\left\lceil\log(1/\delta)\right\rceil}{n}}
 \]
 with $\frak c$ denoting either $2\mathrm e$ or $16\mathrm e^2/(3 \sqrt 3)$ depending on whether the sample is contaminated or not. For $j=1,2,4$, consider the event 
 \[
E_j = \left\{ | \hat{\theta}_{j} -\theta_{j} | \leq c_n(\delta) \right\}\, .
\]
On the one hand, from Propositions~\eqref{prop:medians-degenerate} and~\ref{prop:medians-outliers}, $\mathrm{pr}(E_1\cap E_2\cap E_4)\geq 1-3\delta$ and $\mathrm{pr}_{\mathcal{O}\cup\mathcal{I}}(E_1\cap E_2\cap E_4)\geq 1-3\delta$ depending on the value of $\frak c$. 

On the other hand,
\begin{align*}
    2^{\hat{\xi}}-2^\xi &= \frac{\hat\theta_{4} - \hat\theta_{2} }{\hat\theta_{2} - \hat\theta_{1} }- \frac{\theta_{4} - \theta_{2} }{\theta_{2} - \theta_{1} }\\
    &= \frac{(\hat\theta_{4} - \hat\theta_{2})(\theta_{2} - \theta_{1})-(\theta_{4} - \theta_{2})(\hat\theta_{2} - \hat\theta_{1})}{(\hat\theta_{2} - \hat\theta_{1})(\theta_{2} - \theta_{1})}\\
    &= \frac{(\hat\theta_{4} - \theta_{4})(\theta_{2} - \theta_{1})-(\hat\theta_{2}-\theta_{2})(\theta_{4}-\theta_{1})+(\hat\theta_{1}-\theta_{1})(\theta_{4}-\theta_{2})}{(\hat\theta_{2} - \hat\theta_{1})(\theta_{2} - \theta_{1})}\, .
\end{align*}
On the event $E_1\cap E_2\cap E_4$,
\begin{align*}
   | 2^{\hat{\xi}}-2^\xi |&\leq \frac{2c_n(\delta)(\theta_{4} - \theta_{1})}{(\hat\theta_{2} - \hat\theta_{1})(\theta_{2} - \theta_{1})}=\frac{2c_n(\delta)(2^\xi +1)}{\hat\theta_{2} - \hat\theta_{1}}\, .
\end{align*}
Using that for $a,b>0$, 
\[
|\log(a)-\log(b)|\leq \frac{1}{2}|a-b|\left(\frac{1}{a}+\frac{1}{b}\right) \, , 
\]
we obtain that, with probability larger than $1-3\delta$,
\begin{align*}
    | \hat\xi-\xi|&\leq \frac{c_n(\delta)}{\log(2)(\hat\theta_{2} - \hat\theta_{1})}(2^\xi+1)(2^{-\hat{\xi}}+2^{-\xi}) \, . 
    \end{align*}
\end{proof}

\appendix

\section{Concentration bounds under exchangeability and negative association}
\label{sec:CNA}

\begin{definition}\label{def:NA}
A sequence of real-valued random variables $(X_1,\dots,X_n)$ is said to be negatively associated  if for all subset $A\subset [n]$, and for all (coordinate-wise) non-decreasing functions $f:\mathbb{R}^{|A|}\to \mathbb{R}$ and $g:\mathbb{R}^{n-|A|}\to \mathbb{R}$ such that the expectations below are well-defined, one has
\[
E\left[f\left(X_A\right) g\left(X_{[n]\setminus A}\right)\right]\leq E\left[f\left(X_A\right) \right]E\left[ g\left(X_{[n]\setminus A}\right)\right]\, .
\]
\end{definition}

\begin{definition}\label{def:CNA}
A sequence of real-valued random variables $(X_1,\dots,X_n)$ is said to be conditionally negatively associated (CNA) if for all $S\subset [n]$ and $A\subset S$, and for all non-decreasing functions $f:\mathbb{R}^{|A|}\to \mathbb{R}$ and $g:\mathbb{R}^{|S\setminus A|}\to \mathbb{R}$ such that the expectations below are well-defined, one has
\[
E\left[f\left( X_A\right) g\left( X_{S\setminus A}\right) \,\big|\, X_S\right]\leq E\left[f\left( X_A\right)\,\big|\, X_S \right]E\left[ g\left( X_{S\setminus A}\right)\,\big|\, X_S\right]\, .
\]
In other words, the sequence is CNA if it is negatively associated and all conditionalizations are negatively associated.
\end{definition}

A immediate consequence is that sums of negatively associated random variables concentrate at least as well as sums of independent random variables with the same marginals. More precisely, the Laplace transform can be bounded by the product of marginal transforms: if $(X_1n\dots,XX_n)$ is NA, then for all $\lambda\in\mathbb{R}$,
\[
E\left[\mathrm{e}^{\lambda\sum_{i=1}^n X_i}\right]\leq \prod_{i=1}^nE\left[\mathrm{e}^{\lambda X_i}\right]\, .
\]
Let us also mention an important proposition of NA.

\begin{proposition}\label{proposition:NA}
If $(X_1,\dots,X_n)$ is NA, if $A_1,\dots,A_k$ are disjoint subsets of $[n]$, and if $h_1,\dots,h_k$ are non-decreasing functions defined on $\mathbb{R}^{|A_1|},\dots, \mathbb{R}^{|A_k|}$ respectively, then the sequence $\left(h_1(X_{A_1}),\dots, h_k(X_{A_k})\right)$ is NA.
\end{proposition}

The following result generalizes Proposition~\ref{prop:medians-degenerate} to the case where the sample $(X_1,\dots,X_n)$ is exchangeable and CNA. For simplicity, we only prove it in the non-degenerate case ($q=1$). Notice that the symmetric function $\Psi$ needs to be non-decreasing. Under this additional assumptions, the same bounds hold, with probability $1-2\delta$ instead of $1-\delta$.

\begin{proposition}\label{prop:medians-CNA}
Let $(X_1, \dots, X_n)$ be an exchangeable CNA sample, and , for $m\in [n]$, let $\Psi:\mathbb{R}^m\to \mathbb{R}$ be a non-decreasing symmetric function such that $E\left[\Psi(X_1,\dots,X_m)\right]=\theta$ and $\mathrm{var}\left(\Psi(X_1,\dots,X_m)\right)<\infty$. For all $1\leq k\leq m$, let 
\[
v_k = \mathrm{var}\left(E\left[ \Psi(X_1,\ldots, X_{m})\, |\, X_1,\dots,X_k\right]\right)\, .
\]
Then, for all $\delta \in [\mathrm{e}^{\lfloor n/m\rfloor} ,1)$, 
the median-of-means estimator $\hat{\theta}_n$ defined by \eqref{est:medianofmeans} satisfies 
\begin{equation}\label{eq:prop-median-sub-gaussian-CNA}
\mathrm{pr}\left( | \hat{\theta}_n -\theta | >2e\sqrt{\frac{2mv_m \left\lceil \log(1/\delta)\right\rceil}{n}}\right)\leq 2\delta\, ,
\end{equation}
and
\begin{equation}\label{eq:prop-median-sub-gamma-CNA}
\mathrm{pr}\left( | \hat{\theta}_n -\theta | >2e\sqrt{\frac{2m^2v_1 \left\lceil \log(1/\delta)\right\rceil}{n}+\frac{4m^3v_m \left\lceil \log(1/\delta)\right\rceil^2}{n^2}}\right)\leq 2\delta\, .
\end{equation}
\end{proposition}

\begin{proof}[Proof of Proposition~\ref{prop:medians-CNA}]
Let $t=\min\{t_1,t_2\}$, with 
\[
t_1=2e\sqrt{\frac{2mv_m \left\lceil \log(1/\delta)\right\rceil}{n}}
\]
and
\[
t_2=2e\sqrt{\frac{2m^2v_1 \left\lceil \log(1/\delta)\right\rceil}{n}+\frac{4m^3v_m \left\lceil \log(1/\delta)\right\rceil^2}{n^2}}\, .
\]
To deal with monotone events, we first have to decompose the absolute value:
\[
\mathrm{pr}\left(  |\hat{\theta}_{n} -\theta| >t\right) =\mathrm{pr}\left(  \hat{\theta}_{n} -\theta >t\right) +\mathrm{pr}\left(  \hat{\theta}_{n} -\theta <-t\right)\, .
\]
Let us show that both terms on the right-hand side above are less than $\delta$. We only detail the argument for $\mathrm{pr}\left(  \hat{\theta}_{n} -\theta >t\right)$. The term $\mathrm{pr}\left(  \hat{\theta}_{n} -\theta <-t\right)$ can be treated similarly. We have
\[
\mathrm{pr}\left(  \hat{\theta}_{n} -\theta >t\right)\leq \mathrm{pr}\left(\sum_{j=1}^K Y_j\geq K/2\right)\, ,
\]
with $Y_j=\mathbf{1}_{\left\{ \hat{\theta}^{(j)}-\theta>t\right\}}$. Since $\Psi$ is non-decreasing, the sequence $(Y_1,\dots,Y_K)$ is NA thanks to Property~\ref{proposition:NA}. The bounds of Lemma~\ref{lemma:laplace} (which come from bounds on the Laplace transform) thus apply here as well, and it remains to verify that $E(Y_j)\leq \frac{1}{4\mathrm{e}^2}$, as in the proof of Proposition~\ref{prop:medians-degenerate}. To that aim, it suffices to show that the variance bounds of Lemma~\ref{lemma:variance-block} also holds in the exchangeable CNA setting, under the assumption that $\Psi$ is non-decreasing. We have
\begin{align*}
    \mathrm{var}\left(U\right) &= \frac{1}{{n\choose m}^2}E\left[ \left(\sum_{\substack{A\subset [n]\\ |A|=m}}\left(\Psi\left(X_A\right)-\theta\right)\right)^2\right]\\
    &=\frac{1}{{n\choose m}^2}\sum_{k=0}^m \sum_{\substack{A,B\subset [n]\\ |A|=|B|=m,\, |A\cap B|=k}}E\left[ \left(\Psi\left(X_A\right)-\theta\right) \left(\Psi\left(X_B\right)-\theta\right)\right]\, .
\end{align*}
By exchangeability and CNA, and since $\Psi$ is non-decreasing and symmetric, we have, for all subsets $A$ and $B$ of size $m$ such that $|A\cap B|=k$, 
\begin{align*}
E\left[ \left(\Psi\left(X_A\right)-\theta\right) \left(\Psi\left(X_B\right)-\theta\right)\right]&= E\left[E\left[ \left(\Psi\left(X_A\right)-\theta\right) \left(\Psi\left(X_B\right)-\theta\right)\, \big|\, X_{A\cap B}\right]\right]\\
&\leq E\left[ \left( E\left[\Psi(X_1,\dots,X_m)\,\big|\, X_1,\dots,X_k\right]-\theta\right)^2\right]\, .
\end{align*}
 Since the number of subsets $A$ and $B$ of size $m$ such that $|A\cap B|=k$ is equal to ${n\choose m}{m\choose k}{n-m\choose m-k}$, we arrive at
\begin{equation*}
 \mathrm{var}\left(U\right)\leq \frac{1}{{n\choose m}}\sum_{k=0}^m {m\choose k}{n-m\choose m-k} v_k\, ,
\end{equation*}
where $v_k=E\left[ \left( E\left[\Psi(X_1,\dots,X_m)\,\big|\, X_1,\dots,X_k\right]-\theta\right)^2\right]$. At this point, all the proof of Lemma~\ref{lemma:variance-block} can be repeated, after checking that the proof of Inequality~\eqref{eq:ineq-variance-hoeffding} in~\citep{hoeffding1948} only requires exchangeability. 
\end{proof}

\bibliographystyle{abbrvnat}
\bibliography{biblio}

\end{document}